\documentclass[12pt, reqno]{amsart}
\usepackage{amssymb,amsthm,amsfonts,amsmath}

\usepackage{hyperref}

\newcommand{\eq}{\begin{equation}}
\newcommand{\en}{\end{equation}}

\newcommand{\calP}{\mathcal P}
\newcommand{\Nat}{\Bbb N}

\newcommand{\Z}{{\mathbb Z}}
\newcommand{\q}{{\tt q}}

\def\endpf{\hfill $\Box$ \vskip0.5cm}
\def \proof{\noindent{\it Proof.\ }}

\newtheorem{theorem}{\large Theorem}
\newtheorem{proposition}[theorem] {\large Proposition}

\newtheorem{corollary}[theorem]{\large Corollary}
\newtheorem{lemma}[theorem]{\large Lemma}

\begin{document}
\title{Boundaries from inhomogeneous Bernoulli trials
}
\author{Alexander Gnedin}\thanks{Utrecht University; e-mail A.V.Gnedin@uu.nl}
\date{
\today
\\
}
\maketitle

\begin{abstract}\noindent  
The boundary problem is considered for 
inhomogeneous increasing random walks on the square lattice $\Z_+^2$ with weighted edges.
Explicit solutions are given for some instances related to the classical and generalized number triangles. 
\end{abstract}

\section{Introduction}

\noindent
The  homogeneous Bernoulli processes
all share a  property which may be called {\it lookback similarity\,}: 
if the number of heads $h$ in any first $n$ trials is given,
then independently of the future outcomes the random history of the process  
can be described in some unified way. 
Specifically, the conditional history of homogeneous Bernoulli trials
has the same distribution as a uniformly random permutation of $h$ heads and $t=n-h$ tails.
This property is equivalent to exchangeability, meaning the invariance 
of distribution under   finite
permutations of coordinates.  
 A central structural result
regarding this class of  processes is
de Finetti's theorem, which  asserts that the exchangeable trials can be characterized as mixtures of 
the homogeneous Bernoulli trials.


\par More general  coin-tossing processes, in which  probability of  a head in the next trial 
may depend on the history through the counts of heads and tails observed so far,
are divided into other classes of lookback similar processes, 
each class with its own random mechanism of arranging $h$ heads and $t$ tails in succession,
consistently for every value of $n=h+t$. 
Understanding the structure of such classes is important in a variety of contexts, 
including statistical mechanics, urn models, species sampling, random walks on graphs, dimension theory 
of algebras, ergodic theory and others.
The principal steps of the analysis involve identification of the  extremal boundary, which may be characterized as
the set of ergodic processes in a given class, 
as well as construction and decomposition of 
distinguished nonergodic processes, like e.g. P{\'o}lya's urn processes in the exchangeable case.

\par A technical issue which must be resolved to set up the scene 
is the way of specifying
conditional probabilities on histories. 
The backward transition probabilities are rarely available directly. 
Sometimes 
the starting point is a particular 
reference distribution $P^*$ which determines the class
\cite{Malyshev, Irina}, however,  
in many combinatorial and algebraic contexts 
there might be no natural candidate  
with a desirable property of non-degeneracy.
For instance, if the distributions on histories are determined by some symmetry condition,
like equidistribution of paths in a Bratteli diadram (see e.g \cite{KOO}), 
it  takes some effort to construct a nontrivial process.
The approach adopted in this paper 
amounts to defining conditional
distributions on histories  by means of a multiplicative weight function. 
The underlying structure is the lattice ${\mathbb Z}_+^2$ with weighted edges, which we
call a {\it weighted Pascal graph}. When all weights are $1$ the structure is the Pascal
triangle, in which the path-counts (combinatorial {\it dimensions}) are given by the binomial coefficients.
The coin-tossing process
is encoded into a increasing
lattice path
with unit horizontal and vertical  jumps, 
and the boundary problem is connected to  the asymptotic weighted path-enumeration.
For suitable choices of the weight function, the model covers arbitrary  
coin-tossing processes, for which the joint counts of heads and tails
is the sufficient statistic of the history to predict the outcomes of future trials.

\par  We will consider some classes of lookback-similar processes, for which 
the boundary can be determined without direct path-counting.
Our strategy boils down to three steps.
First of all, we exploit a monotonicity idea to show that  the ergodic processes
are naturally parameterized by the probability of a head in the first trial, 
so the boundary is always a subset of the unit interval.
Then we manipulate with admissible transformations of weights to 
construct a parametric family of processes.
The last step is to either verify  that the constructed processes are ergodic,
or to obtain all extremes  by decomposition in ergodic components. 

\par The generalized Pascal triangles
appear sometimes in connection with  statistics of combinatorial structures like partition and 
composition posets. Analysis of measures on the `triangles' has been proved useful to construct 
processes with values in these more complex objects \cite{perm, Gibbs, GO0, GO2}.
In section \ref{alphaneg} we include one example of this kind, which is a new sequential construction 
of exchangeable partition of a countable set.



\section{Weighted Pascal graphs}

\noindent
We shall be dealing with  Markov chains $S=(S_n, ~n\geq 0)$ on the square lattice
${\mathbb Z}^2_+$, with edges  directed away from the origin. 
A {\it standard path} starts at $(0,0)$ and each time increments by either $(1,0)$ or $(0,1)$.
We think of the increment
$S_{n}-S_{n-1}= (1,0)$ as a head,  and $S_{n}-S_{n-1}= (0,1)$
as a tail in the $n$th coin-tossing trial. 
The components will be denoted $S_n=(H_n,T_n)$, so $H_n+T_n=n$.

Let $w$ be a positive function on the set of edges.
We call the lattice ${\mathbb Z}_+^2$ with weighted edges 
a {\it weighted Pascal graph}.  
The weights of the
edges connecting $(h,t)\to(h+1,t)$ and $(h,t)\to(h,t+1)$
will be denoted $w_1(h,t)$ and $w_0(h,t)$, respectively.

Define the weight of a path 
connecting two given  grid points  to be the product of weights of edges along the path. 
The sum of weights  of all  standard paths with terminal point $(h,t)$ is  denoted $d(h,t)$ and called {\it dimension},
this quantity is analogous to the partition sum in statistical mechanics.
The dimensions satisfy the backward recursion
$$d(h,t)=w_1(h-1,t) d(h-1,t)+ w_0(h,t-1)d(h,t-1)$$
(where the terms vanish if $h=0$ or $t=0$).
More generally, we define the extended dimension
$d(h,t;h',t')$ as the sum of weights of paths from $(h,t)$ to $(h',t')$. 
A class $\calP=\calP(w)$ of lookback similar distributions for $S$ is defined by
the following 
\vskip0.3cm
\noindent
{\bf conditioning property:} if the Markov chain $S$ starting at $S_0=(0,0)$ visits 
 $(h,t)\in {\mathbb Z}_+^2$ with positive probability, then given
 $S_{h+t}=(h,t)$ the  conditional probability 
of each standard path with endpoint $(h,t)$ is equal to the weight of the path divided by $d(h,t)$.
\vskip0.3cm

The class $\mathcal P$ is  a convex, weakly compact set with the property of uniqueness of 
representation of each $P\in{\mathcal P}$ as a convex mixture of the extreme elements of $\mathcal P$.
The set of extreme points of $\mathcal P$ is the {\it extremal boundary}  denoted  ${\rm ext}{\mathcal P}$.
The extremes  are 
 characterized as ergodic measures $P\in {\mathcal P}$,
for which every tail event of the process $S$ has $P$-probability zero or one.
Note that the tail sigma-algebra for $S$ coincides with the exchangeable sigma-algebra for 
the sequence of increments of $S$.
The {\it boundary problem} asks one to describe as explicitly as possible the set ${\rm ext}\calP$.

Each $P\in \calP$
is uniquely determined by probabilities of finite standard paths.
These probabilities are conveniently 
encoded into
 a {\it probability function} on ${\mathbb Z}_+^2$ 
$$\phi(h,t):={P(S_{h+t}=(h,t))\over d(h,t)},$$
so that the probability of a standard path terminating at $(h,t)$ is equal to the weight of the path multiplied by 
$\phi(h,t)$. The transition probabilities from $(h,t)$ to $(h+1,t)$ and $(h,t+1)$ are written then as 
$$
p(h,t):={w_1(h,t)\phi(h+1,t)\over \phi(h,t)},~~~q(h,t):={w_0(h,t)\phi(h,t+1)\over \phi(h,t)},
$$
respectively.
The probability functions are characterized as nonnegative solutions
to a forward recursion
\begin{equation}\label{forward-rec}
\phi(h,t)=w_0(h,t) \phi(h,t+1)+w_1(h,t)\phi(h+1,t),
\end{equation} 
with the normalization $\phi(0,0)=1$. Note that 
\begin{equation}\label{dist-S}
\sum_{h=0}^n \phi(h,n-h)d(h,n-h)=1,
\end{equation}
since the terms make up the distribution of $S_n$.
We mention some intepretations and aspects of the boundary problem.

\vskip0.3cm
\noindent
{\bf 1.} For the Pascal triangle  $w\equiv 1$, and the recursion (\ref{forward-rec})
goes back to Hausdorff \cite{Haus}.
Positivity of $\phi$ in this case means that
 the sequence $\phi(\,\cdot,\,0)$ is completely monotone, hence by the celebrated Hausdorff's 
theorem  it is uniquely representable as a sequence of moments of a probability measure on $[0,1]$.

A similar connection exists in the general case too.
To that end,  
observe that the bivariate array $\phi$ is obtainable by weighted differencing of the sequence
$\phi(\,\cdot\,,0)$, for instance
$\phi(n,1)=(\phi(n,0)-w_1(n,0)\phi(n+1,0))/w_2(n,0)$.
Thus
solving (\ref{forward-rec}) means finding all sequences  $\phi(\,\cdot\,,0)$ for which the weighted 
differences are non-negative. We will show in the next section that ${\rm ext}\calP$ is homeomorphic
to a subset ${\mathcal E}\subset[0,1]$ via $\pi(P)=P(H_1=1)$, therefore 
the extremal decomposition 
$$\phi(n,0)=\int_{\mathcal E} \phi_\pi(n,0)\mu(d\pi),$$
can be seen as a generalized problem of moments on $\mathcal E$,
with the kernel $\phi_\pi(n,0)$, in place of $\pi^n$ from the classical problem of moments.

\vskip0.3cm
\noindent
{\bf 2.} Finite permutations of $\Nat$ act on infinite paths in $\Z_+^2$ by re-arranging the sequence of increments. 
In these terms, the distributions  $P\in\calP$ can be characterized as measures 
{\it quasi}-invariant  under permutations.
The characteristic cocycle of the action is uniquely determined by the 
condition that if a path  
fragment $(h,t)\to(h+1,t)\to(h+1,t+1)$ 
is switched by transposition to 
$(h,t)\to(h,t+1)\to(h+1,t+1)$, then 
 probability of the path is multiplied 
by $w_0(h,t)w_1(h,t+1)w_1(h,t)^{-1}w_0(h+1,t)^{-1}$.
\vskip0.3cm
\noindent
{\bf 3.} By a well known recipe, the extreme $\phi$ are representable  as limits of the `Martin kernel'
\begin{equation}\label{dim-rat}
\phi(h,t)=\lim {{\rm dim}(h,t;h',t')\over {\rm dim}(h',t')},
\end{equation}
where $|h'+t'|\to\infty$, and $(h',t')$ vary in some way to ensure convergence of the ratios for all $(h,t)$.
This relates the boundary problem to the asymptotic weighted path enumeration.
\vskip0.3cm
\noindent
{\bf 4.} Call $P^*\in\calP$ {\it fully supported}, 
if it gives positive probability to every finite path, or, equivalently, the corresponding probability function $\phi^*$
is strictly positive. In principle, all
other probability functions are obtainable then from $P^*$ by a change of measure 
$\phi=\psi\phi^*$ where $\psi$ is a $P^*$-harmonic function
satisfying the recursion
$$\psi(h,t)=p^*(h,t)\psi(h+1,t)+q^*(h,t)\psi(h,t+1),$$
and $p^*,q^*$ are the transition probabilities of $S$ under $P^*$.
This is an instance of Doob's $h$-transform,

\vskip0.3cm
\noindent
We say that $P\in\calP$ is finitely supported if it is not fully supported.
For finitely supported $P\in{\rm ext}\calP$ the probability function
is supported by one of the sets   
$$I_m:=\{(h,t): h\leq m\},~~~~~J_m:=\{(h,t): t\leq m\}$$ 
for some  $m\geq 0$.
The case $m=0$ corresponds to two trivial measures, denoted  $Q_{0,\infty}$ and $Q_{\infty,0}$, 
each supported by a single infinite path
$H_n\equiv 0$ and $T_n\equiv 0$, respectively.


\noindent
{\bf 5.}
If the  weights are integers, we  may consider the graph with 
 multiple edges 
and {\it fine paths}, which distinguish among the edges connecting the neighbouring grid points.
The setting is important e.g.
in ergodic theory to define 
the `adic transform' of the path space
\cite{adic1, adic2}.
When the distinction of edges with the same endpoints is made, 
we understand $\calP$ as the class of probability measures
on fine paths, with a more delicate conditioning property that for every $(h,t)$ 
all standard fine paths terminating at $(h,t)$ have the same
probability $\phi(h,t)$.

Pascal graphs with multiple 
edges may arise by  consistent coarsening  the set of nodes in a tree. 
Suppose ${\mathcal T}$ is an infinite rooted tree with  the set of vertices ${\mathcal T}_n$ ($n\geq 0$) 
at distance $n$ from
the root. Suppose  each ${\mathcal T}_n$ is partitioned in $n+1$ nonempty blocks labeled $(0,n), (1,n-1),\dots,(n,0)$,
so that for $h+t=n$ 
each vertex in block $(h,t)\subset {\mathcal T_n}$    is connected by the same number $w_1(h,t)$ of edges to every vertex in 
block $(h+1,t)$,  
and by  the same number $w_0(h,t)$ of edges to every vertex in block $(h,t+1)$.    
Merging all nodes in each block $(h,t)$ 
in a single node we obtain a weighted Pascal graph,
 whose standard  paths, in turn, ramify according to $\mathcal T$.

\vskip0.2cm
\noindent
{\bf Example.} The {\it Stirling-I triangle} has $w_0(h,t)=h+t+1, w_1(h,t)=1$.
The dimensions
${\rm dim}(h,t)$ are unsigned  Stirling number of the first kind. 
The fine paths of length $n$ are in bijection with 
permutations
of $n+1$ integers. 
Write permutation $\pi_n$ of $[n]:=\{1,\ldots, n\}$ in the one-row notation, like e.g.
${\bf6}~{\bf4}~5~{\bf1}~2~3$. The boldfaced elements  are (lower) records
(a record is  a number smaller than all numbers to the left of it, if any). 
A permutation $\pi_{n+1}$ of $[n+1]$ extending 
$\pi_n$ is  obtained   by inserting integer $n+1$ in one of $n+1$ possible positions, e.g.
${\it 7} {6} {4} 5 {1} 2 3,~~~~{6} {\it 7} {4} 5 {1} 2 3,~~~\cdots~~~,{6} {4} 5 {1} 2 3 {\it 7}.$
The extension organizes permutations of integers $1,2,\dots$ in a tree, 
in which each $\pi_n$ has $n+1$ followers.
Now suppose permutations of $[n]$ are classified by the number of  records.
Each $\pi_n$ with $h$ records ($1\leq h\leq n$) 
is followed by a sole permutation of $[n+1]$ with $h+1$ records, obtained by inserting
$n+1$ in the leftmost position of $\pi_n$, and $n$ permutations with $h$ records. 
Assigning label $(h,t)$ to the class of permutations of $[h+t+1]$ with $h+1$ records,
the classification of permutations by the number of records 
is then captured by 
the Stirling-I triangle.

The same multiplicities  appear when permutations are classified by the number of cycles. 
Arrange elements in each cycle in the clockwise cyclic order. 
Then extending $\pi_n$ amounts to either inserting $n+1$ in some cycle clockwise next to any of 
the integers $1,\ldots, n$,
or starting a new singleton cycle with $n+1$.

\section{Probability of the first head} 
\label{gener}

As we have seen, each  measure $P\in\calP$ is uniquely determined by 
$\phi(\,\cdot\,,0)$, that is by the  sequence of probabilities 
$P(H_n=n)={\rm dim}(n,0)\phi(n,0), n\geq 0$. 
We aim now to show that the probability of the first head 
$$\pi=\pi(P):=P(H_1=1)$$ 
parameterizes $P\in{\rm ext}\calP$.

Every $P\in\calP$ conditioned on $S_{h+t}=(h,t)$
coincides, as a measure 
on the set of  standard paths terminating at $(h,t)$, with the same  
 {\it elementary measure} $Q_{h,t}$ determined by the Martin kernel in (\ref{dim-rat}).
A path $(h'_n,t'_n), n\geq 0,$ is called {\it regular} if 
the elementary measures $Q_{h_n',t_n'}$ converge
weakly along the path, in which case the limit  necessarily belongs to $\calP$. 
We may also say that the limit measure is  induced by the path.
Denote ${\rm ext}^\circ\calP$ the set of measures induced by   regular paths.
By some general theory it is  known that
for every $P\in {\mathcal P}$,  the set of regular paths 
has $P$-probability 1, and that
$P\in {\rm ext}{\mathcal P}$ if and only if the set of regular paths 
that induce $P$ has $P$-probability 1. In particular, ${\rm ext}^\circ\calP\supset{\rm ext}\calP$.

\par For $P,P'\in \calP$ we say that $P$ is {\it stochastically larger} than $P'$ if
$
P(H_n\geq h)\geq P'(H_n\geq h)
$ 
for all $(h,n-h)\in{\mathbb Z}^2_+$. This defines a partial order on $\calP$, and we will  show that
it restricts as a total order on ${\rm ext}^\circ\calP$.

\begin{lemma}\label{L2} Let $(h_n,n-h_n)$ and $(h'_n,n-h'_n)$, $n\geq 0$, be two regular paths which induce $P$ and $P'$, respectively.
Then  
$P$ is stochastically larger than $P'$ iff
$h_n\geq h_n'$
for all large enough $n$. In this case the measures are distinct if and only if  $P(H_1=1)>P'(H_1=1)$, 
and then in fact
$P(H_n=n)>P'(H_n=n)$ for all $n\geq0$.
\end{lemma}
\proof Choose  $0\leq h'<h\leq n$ and consider the elementary measures $Q_{h_n,n-h_n}$ and $Q_{h'_n,n-h'_n}$ 
as the laws of two Markov chains $S$ and $S'$ (respectively) with $n$ moves. 
We define  a coupling of independent $S$ and  $S'$ by running the chains in the  backward time. Start $S,S'$ simultaneously 
at states
$(h_n,n-h_n)$ and $(h'_n,n-h_n')$, respectively, and let them running independently until the first time $\tau$  
when they meet at the same state. From time $\tau$ on let $S'$ coincide with $S$: this does not 
affect the marginal distributions of $S'$ since both chains have the same transition probabilities.
Because the number of heads  each time decrements by $0$ or $1$, by this coupling $S$ always has at least 
as many heads as $S'$. 
It follows that $P(H_m=m)\geq P'(H_m=m)$ for $m\leq n$.

\par Now suppose the elementary measures converge along paths $(h_n,n-h_n)$ and $(h_n',n-h_n')$ to distinct 
$P$ and $P'$, respectively. Two paths either have infinitely many intersections or the inequality between 
$h_n$ and $h_n'$ is definite for large enough $n$.
If eventually $h_n>h_n'$, then by the coupling argument
we conclude that $P$ is strictly stochastically
larger than  $P'$, in which case $P(H_n=n)>P'(H_n=n)$, as is readily justified.
\endpf

\begin{proposition}\label{hom} The set ${\rm ext}^\circ {\mathcal P}$ is weakly compact and  
$P\mapsto \pi(P)$
is a homeomorphism of ${\rm ext}^\circ{\mathcal P}$ onto  a subset of $[0,1]$. 
\end{proposition}
\proof
Assume $P_j\in {\rm ext}^\circ{\mathcal P}$ converge weakly to some $P\in {\mathcal P}$.
By Lemma \ref{L2} it is enough to consider the case when the sequence $P_j$ is monotonic in the stochastic order,
say increasing.  For each $P_j$ fix a path inducing it.
Then it is always possible to 
choose a path  which has the last intersection with
the path inducing $P_j$ at some time $n_j$, where $n_j\to\infty$.
It is readily seen that the path induces $P$, hence $P\in {\rm ext}^\circ\calP$.
Thus the sequential boundary is compact.
It remains to note that
$P\mapsto P(H_1=1)$ is a continuous strictly increasing function on ${\rm ext}^\circ\calP$.
\endpf

Henceforth the boundary can be identified with ${\mathcal E}:= \{\pi(P):~ P\in{\rm ext}\calP\}$. 
The set $\mathcal E$ may be fairly arbitrary, but 
in any case it contains the endpoints $0$ and $1$ that correspond to
the trivial measures. 

\par If there exists a finitely supported measure with probability function strictly positive on $I_m$
and zero otherwise, we denote this measure $Q_{m,\infty}$. 
If $Q_{m,\infty}$ exists, it is extreme and induced by the path $(m,n-m), n\to\infty$.
Likewise, a measure 
with probability function strictly positive on $J_m$
and zero otherwise is denoted $Q_{\infty,m}$, $m\geq 0$.
In some cases the boundary is comprised of only finitely supported measures:

\begin{lemma}\label{discr} 
Suppose the set of values of $m$ for which the finitely supported measure $Q_{m,\infty}$ exists
is an infinite increasing sequence. If along this sequence $\pi(Q_{m,\infty})\to 1$ then
these finitely supported measures and the trivial measure $Q_{\infty,0}$ exhaust the boundary.
\end{lemma}

In some cases the boundary is homeomorphic to $[0,1]$:
\begin{lemma}\label{conti}{\rm \cite{Gibbs}} Supppose there is a sequence of positive constants 
$(c_n)$ with $c_n\to\infty$,
and for each $s\in[0,\infty]$  there is a $P_s\in\calP$ which satisfies $P_s(H_n/c_n\to s)=1$.
Suppose $s\mapsto P_s$ is a continuous injection from $[0,\infty]$ to $\calP$, with $0$ and $\infty$
corresponding to the trivial measures. Then a path $(h_n,n-h_n)$ is regular if and only $h_n/c_n\to s$ for 
some $s\in[0,\infty]$, in which case the path induces $P_s$. Moreover,
${\rm ext}^\circ\calP={\rm ext}\calP=\{P_s,s\in[0,\infty]\}.$
\end{lemma}
In the situation of the last lemma the parameter $s$ has the meaning of the asymptotic 
frequency of heads on the $c_n$-scale. 
If $c_n=n$ can be chosen, like in the case of exchangeable trials,
then we should take for the range of $s$ a finite interval (and not $[0,\infty]$).

\section{Transformations of weights}

\noindent
We shall approach the question of costructing some lookback similar processes
through admissible transformations of $w$, which 
change dimensions but do not affect the conditioning property.

\begin{proposition} Two weight functions $w$ and $w'$  yield the same 
class of lookback similar distributions, i.e. 
 $\calP(w)=\calP(w')$, if and only if    
there exists 
a positive function $f$ on ${\mathbb Z}_+^2$ such that for every edge $s\to s'$
$$
w'(s,s')= w(s,s')f(s)/f(s').
$$
\end{proposition}
\proof For the `if' part, the product of weights $w'(s,s')$ along a path depends only on the endpoints of the path,
hence the conditioning property is the same as for $w$. 
The `only if' part follows by induction on the length of the path. 
\endpf
A family of admissible transformations has the form 
\begin{equation}\label{weighttrans}
w'_0(h,t)=g_0(t) g(h+t)w_0(h,t), ~~~w'_1(h,t)=g_1(h)g(h+t)w_1(h,t),
\end{equation}
with arbitrary positive $g,g_0,g_1$.
We shall use three  instances of (\ref{weighttrans}): 
\begin{itemize}
\item[(i)] multiplying $w_0$ by a positive function of $t$,
\item[(ii)] multiplying $w_1$ by a positive function of $h$,
\item[(iii)] multiplying both $w_0$ and $w_1$ by the same positive function of $h+t$.
\end{itemize}

Call $w$   {\it balanced} if $w_0(h,t)+w_1(h,t)=\sigma(h+t)$ for some function $\sigma$.
For instance, the Pascal and Stirling-I triangles are balanced.
If $w$ is balanced then there is a fully supported $P\in\calP$  
with transition probabilities
\begin{eqnarray}\label{tr-pr}
p(h,t)=  {w_0(h,t)\over \sigma(h+t)},~~~
q(h,t)=  {w_1(h,t)\over \sigma(h+t)}.
\end{eqnarray}   
Conversely, for fully supported $P\in{\calP}$ the transition probabilities $p,\, q$ themselve
define an equivalent  weight function.
Thus 
\begin{itemize}
\item
finding a fully supported $P$ is equivalent to 
transforming  $w$ to a balanced
weight function $w'$.
\end{itemize}

A variation of this rule leads to finitely supported measures.
For instance, if in (\ref{weighttrans}) we take $g_0$ positive on $I_m$ and zero otherwise,
then the resulting $w'$ is not a strictly positive weight. However, if
the balance condition holds, then $w'$ defines a finitely supported measure $Q_{m,\infty}$ via (\ref{tr-pr}).

Let us review from the above positions the prototypical example of the Pascal triangle.
\vskip0.2cm
\noindent
{\bf The Pascal triangle.} 
The weight is $w\equiv 1$ and
the dimension function is  $d(h,t)={h+t\choose h}$.
The conditioning property identifies  $\calP$ as the family of processes of 
exchangeable  trials. 
We will not change $\calP$ if instead we choose $w_1>0$ to be an arbitrary function of 
$h$, and $w_0>0$ to be an arbitrary function of $t$, but imposing the balance condition to construct a measure
we are  restricted to linear functions, as is easily checked.

Transforming weights to 
$$w_0'(h,t)=1, w_1'(h,t)=\theta$$
with $\theta>0$ yields balanced weight functions, hence some measures $P_\theta\in\calP$.
The constructed process is
a sequence of
homogeneous Bernoulli trials with 
$$p(h,t)={\theta\over 1+\theta},~~~~q(h,t)= {1\over 1+\theta},$$
By the law of large numbers $P_\theta(H_n/n\to \pi)=1$ for $\pi=\theta/(\theta+1)$.
The family $P_\theta$ is continuous, and the trivial measures appear as $\theta\to 0$, respectively 
$\theta\to\infty$.
By Lemma \ref{conti} all extremes are found, and ${\mathcal E}=[0,1]$. This is de Finetti's theorem.

Note that the kernel in (\ref{dim-rat}) is
$${{\rm dim}(h,n-h;h',n'-h')\over{\rm dim}(h',n'-h')}= {n'-n\choose h'-h}{\bigg/}{n'\choose h'}.$$
From the criterion of path regularity 
we see that
the ratios of binomial coefficients
converge as $n'\to\infty$ (for all $n$ and $0\leq h\leq n$) if and only if $h'/n'\to\pi$ 
for some $\pi\in [0,1]$. The latter fact can be derived 
directly from the formula for binomial coefficients, and we include 
it here just to show how conclusions on aspects of the asymptotic behaviour of these and other
combinatorial numbers can be made.

For $a,b>0$, another transformation of weights
with $g_1(t)=a+h, g_0(t)=b+t$
results  in the balanced weight function
$$w_1''(h,t)=h+a,~~~ w_0''(h,t)=t+b.$$
This gives
a Markov chain with
$$\phi(h,t)={(a)_h (b)_t\over(a+b)_{h+t}}$$
(where $(x)_m$ is the rising factorial). This is known
as the P{\'o}lya urn process with starting configuration $(a,b)$. 
The representation as mixture of homogeneous Bernoulli processes follows 
from the fact that the limit law of 
$H_n/n$ is  the beta distribution with density
$${\Gamma(a)\Gamma(b)\over \Gamma(a+b)} \pi^{a-1}(1-\pi)^{b-1},~~~\pi\in [0,1].$$

\section{Generalized Stirling triangles}

\noindent
Following Kerov \cite{Kerov}, for two real sequences $(a_n), (b_h)$ 
consider a weighted Pascal graph with  
$$w_0(h,t)=a_{h+t}+b_h, ~~w_1(h,t)=1,$$
where
$a_{h+t}+b_h>0$ for $(h,t)\in\Z_+^2$.
The dimensions are 
 {\it the generalized Stirling numbers}.

Introduce polynomials in $\theta$
$$A_n(\theta)=(\theta+a_0)\dots(\theta+a_{n-1}), ~~~B_n(\theta)=(\theta-b_0)\dots(\theta-b_{n-1}).$$
Transforming $w$  as
$$w_1'(h,t)=\theta-b_h, ~~~w_0'(h,t)=a_{h+t}+b_h,$$
we obtain $w'$ which satisfies the balance condition with $\sigma(n)=\theta+a_{n}$, 
and we obtain a function
\begin{equation}\label{phi-theta}
\phi_\theta(h,t)= {B_h(\theta)\over A_{h+t}(\theta)}
\end{equation}
satisfying (\ref{forward-rec}).
From (\ref{dist-S}) follows that dimensions are 
 the transition coefficients between two polynomial bases,
$$A_n(\theta)=\sum_{h=0}^n d(h,n-h) B_h(\theta),$$
which justifies the name {\it generalized Stirling triangle}.
In the case of the  Stirling-I numbers, the transition is from powers to
factorial powers of $\theta$, and for  Stirling-II other way round.

Using (\ref{phi-theta}) 
yields transition probabilities for some $P_\theta\in\calP$,
$$p(h,t)={\theta-b_h\over \theta+a_{h+t}}, ~~~q(h,t)={a_{h+t}+b_h\over \theta+a_{h+t}},$$
provided that care of positivity is taken.
For any bounded range of $(h,t)$ we can indeed achieve that $w',p,q$ are all positive by choosing large enough 
$\theta$. If $\sup_h b_h<\infty$ then a fully supported measure $P_\theta$ is defined by choosing $\theta>\sup_h b_h$.

\begin{proposition}\label{bound-b} 
For every 
$m\geq 0$ satisfying
$b_m>b_h$ for $h<m$, 
there is a finitely supported $Q_{m,\infty}\in\calP$
with probability of the first head
$$\pi= {b_m-b_0\over b_m+a_0}.$$
If $\sup_h b_h=\infty$ then the boundary  $\mathcal E$ of the generalized Stirling triangle is 
this increasing sequence together with  its  accumulation point at $\pi=1$.
\end{proposition}
\begin{proof} Suppose $b_m\leq b_{m-1}$. The weight of the path connecting $(m,m)$ with $(m,n-m)$ is $\prod_{j=m}^{n-1}(a_j+b_m)$,
while the total weight of all standard paths with endpoint $(m,n-m)$ is estimated from below as 
$(n-m)\prod_{j=m}^{n-1}(a_j+b_{m-1})$. Since the latter is of the higher order, the path $(n-m,m)$ cannot
induce the measure $Q_{m,\infty}$, hence this measure does not exist. On the other hand,
if $b_m$ is strictly larger than $b_0,\dots,b_{m-1}$ then by a similar argument $Q_{m,\infty}$ is the measure
$P_\theta$ with $\theta=b_m$. Application of Lemma \ref{discr} ends the proof.
\end{proof}

\vskip0.2cm
\noindent
{\bf Example.} The classical Stirling-II triangle has $a_n\equiv 0$ and $b_{h}=h+1$.
For $m\geq 0$ the extreme measure $Q_{m,\infty}$ defines the familiar coupon-collectors'
process with $m+1$ distinct coupons, in which a head appears each time a new coupon is sampled.

\vskip0.2cm
If $\sup_h b_h<\infty$ the situation is more complex. We shall consider special cases of
{\it the generalized Stirling-I triangles} with $b_h\equiv 0$, {\it the generalized Stirling-II triangles} with
$a_n\equiv 0$, and two more specific  parametric families.

\subsection{The linear weights}
Most results of this section are taken from \cite{Gibbs}.
Suppose $a_n=n+1, ~b_h=- \alpha(h+1)$ with parameter $\alpha<1$.
Each $P$ corresponds to a exchangeable Gibbs partition of ${\mathbb N}$ of type $\alpha$ \cite{Gibbs}.
For suitable $\theta$
a probability function is defined by
\begin{equation}\label{CRP}
\phi_\theta(h,t)={(\theta+\alpha)(\theta+2\alpha)\dots (\theta+h\alpha)\over(\theta+1)_{h+t}},
\end{equation}
This Markov chain has transition probabilities
$$
p(h,t)={\theta+\alpha (h+1)\over h+t+1+\theta},~~~~q(h,t)={h+t+1-\alpha (h+1)\over h+t+1+\theta}.
$$
To ensure positivity we must require that $\theta\geq 0$ for $0\leq\alpha<1$,
and $-\theta/\alpha\in{\mathbb Z}_+$ for $\alpha<0$.

The number of heads $H_{n+1}$ coincides with  the number of blocks 
in a partition-valued process known as the
`Chinese restaurant process' (CRP)
\cite{CSP}. In more detail, the CRP construction is the following rule. Start at time 
$0$ with the unique partition of the set $\{1\}$.
Suppose at step $n-1$ there is a partition of $\{1,\dots,n\}$ with blocks
of sizes $n_1,\dots,n_{h+1}$, then element $n+1$ joins block $j$ with probability $(n_j-\alpha)/(n+\theta)$,
and starts a new block with probability $(\theta+(h+1)\alpha)/(n+\theta)$.

\subsubsection{The case $\alpha<0$.} \label{alphaneg}
This case is covered by Proposition \ref{bound-b}.
The possible values are $\theta=-\alpha m$, and the boundary is comprised of 
$Q_{m,\infty}, m\geq 0,$ and $Q_{\infty,0}$.

For $\alpha=-1$ 
there is a nonlinear transformation of weights
by multiplications with $g_0(t)=t+\gamma,~g_1(h)=(h+1)(h+1-\gamma)$, where $0<\gamma<1$, leading to the balanced 
weight function
$$w_1''(h,t)=(h+t+1)(h+t+1+\gamma),~~w_0''(h,t)=(h+1)(h+1-\gamma).$$
This yields a nonergodic process with transition probabilities
\begin{equation}\label{trpr1}
p(h,t)={(h+1)(h+1-\gamma)\over(h+t+1)(h+t+1+\gamma)},~~q(h,t)={(2h+t+2)(t+\gamma)\over (h+t+1)(h+t+1+\gamma)}.
\end{equation}
In the same way as in \cite{Gibbs}, there is an exchangeable 
partition-valued process, for which $p(h,t)$ is the probability of a  new 
 block at time $h+t+1$. 
The measure $Q_{m,\infty}$ enters the decomposition over the boundary with a weight equal to the 
probability of $H_n\to m$; this is also the weight by representing  the corresponding frequencies of exchangeable 
partition as a mixture of symmetric Dirichlet distributions.

The constructed exchangeable partition is
a new version of the CRP, with the rule: if at time $n-1$ the partition 
of $\{1,\dots,n\}$ has blocks of sizes $n_1,\dots,n_{h+1}$, then element 
$n+1$ starts a new block with probability $p(h,n-h)={(h+1)(h+1-\gamma)\over n(n+\gamma)} $  (as in (\ref{trpr1})),
and  joins block $j$ with probability  ${(n_j+1)(n-h-1+\gamma)\over n(n+\gamma)},~ j=1,\dots,h+1$.

\subsubsection{The case $\alpha=0$.} This is the Stirling-I triangle,
  closely related to random permutations and other logarithmic 
combinatorial structures \cite{ABT}.

 For $0<\theta<\infty$,  $S$ is the process of Bernoulli trials with probability of a head $\theta/(\theta+n+1)$ at 
trial $n\geq 0$. By the strong law of large numbers 
the measures (\ref{CRP}) satisfy 
$$P_\theta(H_n/\log n\to\theta)=1,$$
thus by Lemma  \ref{discr} the measures are extreme and exhaust the boundary, 
which may be parameterized by $\pi={\theta\over\theta+1}\in [0,1]$.

\subsubsection{The case $0<\alpha<1$.}\label{alphaL1} This case is related to the excursion theory of 
recurrent continuous-time 
Markov processes, like Brownian motion in the case $\alpha=1/2$ \cite{CSP}.

The $P_\theta$'s are not extreme, since
$H_n/n^\alpha$ has a nontrivial limit distribution. 
The boundary is continuous, ${\mathcal E}=[0,1]$, as in Lemma \ref{conti},
and the nontrivial extreme measures are obtained
by conditioning any $P_\theta$ with $-\alpha<\theta<\infty$ on $H_n/n^\alpha\to s$ for $0<s<\infty$. 
See \cite{Gibbs} for details.

\subsection{Generalized Stirling-I} 
Suppose $b_h\equiv 0$ and $a_n>0$, so $w_0(h,t)=a_{h+t}, w_1(h,t)$.
Under $P_\theta$ the process $S$
 is a process of inhomogeneous Bernoulli trials, sometimes called space-time random walk.
The probability function of $P_\theta$ is
$$\phi_\theta(h,t)={\theta^h\over(\theta+a_0)\dots(\theta+a_{h+t-1})}.$$

The structure of the boundary was sketched in \cite{Pitman-DeFin} (abstract of a conference talk). 
Later on Kerov \cite{Kerov} stated a conjecture which disagreed with \cite{Pitman-DeFin}.
Here we add 
some details to \cite{Pitman-DeFin}, in particular we confirm that the  
criterion for ${\mathcal E}=[0,1]$ is any of the equivalent conditions 
\begin{equation}\label{divser}
\sum_n {a_n\over (1+a_n)^2}=\infty\quad\Longleftrightarrow \quad\sum_n{\min (a_n,1)\over 1+a_n}=\infty.
\end{equation}
For instance, if $a_n=n^\beta$, then the boundary is $[0,1]$ if  $|\beta|\leq 1$, and it is discrete otherwise.

Any system of weights of the form $w_0'(h,t)=v_0(h+t), w_1'(h,t)=v_1(h+t)$ with $v_0(n)/v_1(n)=a_n$ yields the same
class $\calP$, and for these the criterion (\ref{divser}) assumes the form stressing symmetry between heads and tails:
$$\sum_n {v_0(n)v_1(n)\over (v_0(n)+v_1(n))^2}=\infty.$$

It is convenient to re-denote the transition probabilities under $P_\theta$ as   
$$p(n)={\theta\over \theta+a_{n}}, ~~~q(n)={a_{n}\over \theta+a_{n}},$$
where  $n=h+t$ and $\theta\in[0,\infty]$.

If (\ref{divser}) does not hold then $P_1$ (and any other nontrivial $P_\theta$) is not extreme, because 
the variance of $H_n$ remains bounded as $n\to\infty$, 
and the centered $H_n$'s
converge weakly to a nontrivial distribution. 
Thus $P_\theta\in{\rm ext}\calP$ can only hold if the series diverges.
In the latter case Lemma \ref{conti} has a limited applicability, because in general there might be no
common scaling $c_n\to\infty$ suitable for the full range of $\theta$.

\subsubsection{The case of continuous boundary}
\begin{proposition} If {\rm (\ref{divser})} holds
then 
${\rm ext}\calP=\{P_\theta,\theta\in[0,\infty]\}$,
with ${\mathcal E}=[0,1]$ being the range of 
$\pi={\theta\over a_0+\theta}$.
\end{proposition}
\proof Under  $P_\theta$  the tail sigma-algebra of $S$ is trivial.
This follows by general
Mineka's criterion   
for the tail sigma-algebra generated by the sequence of sums of independent random variables
(see \cite{Mineka}, Theorem 1 and Corollary on p. 169).
Mineka's condition specifies as
$\sum_n  \min(p(n),q(n))=\infty$, and follows from our assumption for every $\theta\notin\{0,\infty\}$.
As $\theta$ varies, $\pi$ runs over the full range $[0,1]$, thus by Proposition \ref{hom}
all extremes are found.
\endpf

See \cite{Baker} for  extensions to more general space-time  lattice walks,
and \cite{aldous} for conditions of triviality of the exchangeable sigma-algebra
for multivalued processes.

\subsubsection{The case of discrete boundary} We assume now that 
$\sum_n p(n)q(n)<\infty$ for $\theta=1$, and take $P^*=P_1$ for the reference measure, which we wish to decompose 
in ergodic components.

Recall that the elementary symmetric  function of degree $k$
in the variables  $x_1,x_2,\dots$  is 
the infinite series 
$${\tt e}_k(x_1,x_2,\dots)=\sum_{i_1<\dots<i_k} x_{i_1}\cdots x_{i_k},$$
This becomes the elementary symmetric polynomial ${\tt e}_k(x_1,\dots,x_n)$
upon substituting $x_m=0$ for $m>n$. 

\par Introduce the odds ratios $r_n:=p(n)/q(n)$.  
The conditional probability given $S_{h+t}=(h,t)$ that $S$
has $h$ heads 
 at times $\{n_1,\dots,n_h\}\subset \{1,\dots, h+t\}$ is 
$$  {r_{n_1}\cdots r_{n_h}\over e_h(r_1,\dots,r_{h+t})}\,.$$

 Let $L_n=\{m\leq n: p(m)\leq q(m)\},   M_n=\{m\leq n: p(m)> q(m)\}, L=\cup_n L_n, M=\cup_n M_n$.   
Since 
$$\sum_n p(n)q(n)<\infty\Longleftrightarrow
\sum_{n\in L} p(n)<\infty {\rm ~~and~~}\sum_{n\in M} q_n<\infty,$$ 
the Borel-Cantelli lemma implies that 
 $P^*$-almost surely $S$ has finitely many $(1,0)$-increments at times $n\in L$ and finitely many $(0,1)$-increments at times  $n\in M$.
The latter means that $S_n=(H_n,T_n)$ is essentially converging, i.e. 
$$(H_n,T_n)-(\#M_n,\#L_n)\to (Z,-Z)\quad P^*{\rm -a.s.},$$
for some integer-valued random variable $Z$.
Let $R$ be the  
range of $Z$; this  is either $\mathbb Z$, 
or a semi-infinite integer interval 
if $M$ or $L$ is finite.

The variable $Z$ is tail-measurable, thus conditioning $P^*$ on the value of $Z$ we obtain a 
countable family of probabilities $\{P_z^*, z\in {\mathbb Z}\}\subset{\mathcal P}$.
Every $P_z^*$ is ergodic, since it is supported by a single class of equivalent paths
which eventually coincide with the path $h_N=\#M_N-z,~t_N=\#L_N+z$.

\begin{proposition} If {\rm (\ref{divser})} does not hold 
then ${\rm ext}{\mathcal P}=\{P_z^*, z\in R\}\cup\{Q_{0,\infty},Q_{\infty,0}\}$.
\end{proposition}
\proof
We wish to prove that the list of extremes is full.
Using the ordering arguments as in section \ref{gener}, it is clear that  if $|h_N-\#M_N|$ is bounded then
$Q_{h_N,N-h_N}$ may converge only to some $P_z^*$. 
By  Proposition \ref{hom},
it is enough to show that  $P_z^*(H_1=1)\to 1$ or $0$ as $z\to+\infty$ or $-\infty$, respectively.
We shall focus on the first relation, the second being analogous.

\par Consider  first the special case $\sum_n p(n)<\infty$, when
$H_n$ converges $P^*$-almost surely to some finite random variable $H$.
Then also $M$ is finite, $\sum_n r_n<\infty$ and  
$$e_k(r_1,r_2,\dots)<\left( \sum_n r_n\right)^k<\infty.$$
Since $M$ is finite, conditioning on a large value of $Z$ is the same as conditioning on a large value of $H$.
Thus it is enough to show that
$$P^*(H_1=1|H=h)={r_1 e_{h-1}(r_2,r_3,\dots)\over e_h(r_1,r_2,\dots)}\to 1,~~~~{\rm as}~h\to\infty.$$
Decomposing $e_h(r_1,r_2,\dots)=r_1e_{h-1}(r_2,r_3,\dots)+e_h(r_2,r_3,\dots)$ we are reduced to checking that
$e_h(r_2,r_3,\dots)/e_{h-1}(r_2,r_3,\dots)\to 0$. The latter follows from term-wise estimates 
$r_{i_1}\dots r_{i_h}\leq s_h r_{i_1}\dots r_{i_{h-1}}$ with $s_h:=\max_{n>h}r_n\to 0$.

\par In the general case, we define $H'$ to be the number of heads at times $n\in L\cup\{ 1\}$. 
By independence of the increments we get exactly as above $P^*(H_1=1|H'=h)\to 1$ as $h\to\infty$.
Now $P^*(H_1=1|Z=z)\to 1$ for $z\to\infty$ follows from this 
and $H'\geq Z$ by conditioning on the number of tails at times $n\in M\setminus\{1\}$.
\endpf
 
The weights in the decomposition of $P^*$ over ${\rm ext}\calP$ are the proibabilities
$$P^*(Z=z)=\sum_{i,j:\, i-j=z} e_i(r_\ell, \ell\in L)e_j(r_m^{-1},m\in M)\prod_{\ell\in L}q(\ell)\prod_{m\in M}p(m).$$
The boundary $\mathcal E$ has the unique accumulation point at $1$ iff $\sum_n p(n)<\infty$,
the unique accumulation point at $0$ iff $\sum_n q(n)<\infty$, 
and if both series diverge but $\sum_n p(n)q(n)<\infty$ then the only accumulation points are $0$ and $1$.

\subsection{Generalized Stirling-II}
The generalized Stirling-II triangle has weight function $w_1(h,t)=1, w_0(h,t)=b_h$, where $b_h>0$.
The measure $P_\theta$ has 
probability function 
$$\phi_\theta(h,t)={(\theta-b_0)\dots(\theta-b_{h-1})\over \theta^{h+t}},$$
where either $\theta>\sup_h b_h$, or $\theta=b_m$ for some $m$ such that $b_m$ is the strict maximum of $b_0,\ldots,b_m$.
We shall assume $\sup_h b_h<\infty$, since the opposite situation is covered by Proposition \ref{bound-b}.
Each such $P_{b_m}$ is the ergodic measure coinciding with $Q_{m,\infty}$, hence  
we focus on $P_\theta$'s with $\theta>\sup_h b_h$.

The process $S$ under $P_\theta$ can be seen as a  coupon-collector's sampling scheme.
The sampling starts at time $0$ with coupon labelled $0$. Each time $n\geq 0$ when coupons $0,\dots,h$  
are in the collection, a new coupon to be labelled 
$h+1$ is drawn with probability $\lambda_h=b_h/\theta$.

Slightly more generally, let $P$ be a probability under which $\xi_0,\xi_1,\dots$ are independent geometric 
variables with
$$P(\xi_h=k)=\lambda_h^{k-1}(1-\lambda_h), ~~~k\in {\mathbb N}.$$
Define an inhomogeneous `renewal process'
$H_n:=\max\{h\in {\mathbb Z}_+:\sum_{j=0}^{h-1} \xi_j\leq n\}$, 
and $T_n:=n-H_n$, $S_n=(H_n,T_n)$.

\begin{lemma}  \label{ttail} The tail sigma-algebra of $S$ is trivial if and only if
 $\sum_{h=0}^\infty \lambda_h=\infty$.
\end{lemma}
\proof  Let ${\mathcal S}_n=\sigma\{S_0,\dots,S_n\}$ and ${\mathcal S}=\cap_{k\geq n}{\mathcal S}_k$ 
be the tail sigma-algebra of $S$.
If the series in $\sum_n\lambda_n$ converges, the probability $P(\lim_{n\to\infty} T_n=0)=\prod_{h=0}^\infty (1-q_h)$
is strictly between $0$ and $1$, hence $\mathcal S$ contains a nontrivial event.
Moreover,  $T_n$ converge to a finite variable $P$-a.s.

Suppose the series diverges. 
Since in 
any case $H_n\to\infty$ $P$-a.s., we can define 
`renewal' times $\tau_h:=\min\{n: H_n=h\}=\xi_0+\dots+\xi_{h-1}$. 
Let ${\mathcal G}_h$ be the sigma-field generated by the events
$\{\tau_h=n\}\cap A$ for $A\in {\mathcal S}_n$ and $n\geq 0$,
and define ${\mathcal G}:=\cap_{h=0}^\infty {\mathcal G}_h$.
We assert that ${\mathcal G}$  coincides with $\mathcal S$,  up to zero events.
Indeed, we have $\tau_n\geq n$, whence ${\mathcal G}_n={\mathcal S}_{\tau_n}\subset {\mathcal S}_n$
and ${\mathcal G}\subset{\mathcal S}$.
In the other direction,  $\{\tau_h<n\}\cap A\in {\mathcal G}_h$ provided that $A\in{\mathcal S}_n$, whence
${\mathcal G}\supset{\mathcal S}$ obtains as $n\to\infty$ then $h\to\infty$.
Mineka's criterion  
(\cite{Mineka}  Theorem 1 and Corollary on p. 169)
becomes
$$\sum_{h=0}^\infty \sum_{i=0}^\infty \min(P(\xi_h=i),P(\xi_h=i+1))=\sum_{h=0}^\infty \lambda_h,$$
from which $\mathcal G$ is trivial.   
\endpf

\begin{proposition}\label{strbo}
If $\sum_h b_h<\infty$ then each $P_\theta$ for $\theta> \sup b_h$ is decomposable,
and ${\rm ext}\calP=\{Q_{m,\infty}: b_m>b_j ~{\rm ~for~}j<m\}\cup\{Q_{\infty,m}: m\geq 0\}$;
then $\mathcal E$
is a discrete set with the only accumulation point at $\pi=1-b_0/( \sup_n b_h)$.

If $\sum_h b_h=\infty$ then
${\rm ext}\calP=\{Q_{m,\infty}: b_m>b_j ~{\rm ~for~}j<m\}\cup\{P_\theta: \theta>\sup_h b_h\}$;
and 
${\mathcal E}$ contains the interval $[1-b_0/( \sup_h b_h), 1]$.
\end{proposition}
\proof If the series converges, then $T_n$ converge weakly under $P_\theta$, with $\theta> \sup b_h$,
to a finite random  variable. Conditioning on the terminal yields $Q_{m,\infty}$.

If the series diverges then every $P_\theta$ with 
$\theta> \sup b_h$ is ergodic by Lemma \ref{ttail}.
\endpf

\noindent
{\bf Remark}
This result disproves Theorem  in \cite{Kerov} in case $\sum_h b_h<\infty$. 
The cited theorem already disagreed with the  instance of $q$-Pascal triangle discussed in \cite{Kerov}.

\subsection{The $\q$-Pascal triangle}

A $\q$-analogue of the Pascal triangle \cite{Kerov, GO1, GO2}
is the generalized Stirling-II triangle with
weights $w_0(h,t)=\q^h, ~w_1(h,t)=1$, and  dimension
$d(h,t)={h+t\choose t}_\q$ equal to the $\q$-binomial coefficient.
For integer $q$ equal to a power of a prime number, the lookback similar processes 
on the $q$-Pascal triangle encode homogeneous measures on Grassmanians
and on more general flag manifolds in the infinite-dimensional space $({\mathbb F}_q)^\infty$ over the
Galois field ${\mathbb F}_\q$ \cite{GO2, GO3}.

Suppose $0<\q<1$.
By Proposition \ref{strbo} the only ergodic measures are $Q_{m,\infty}$, 
which are given explicitly by the probability function
$$\phi(h,t)= \q^{(m-t)h}\prod_{j=0}^{t-1}(1-\q^{m-j}), ~~~~m=1,2,\dots.$$
The boundary is ${\mathcal E}=\{\q^{m}, m\geq 0\}\cup\{0\}$.
The related problem of moments has a curious form:

\begin{corollary}
A sequence $\phi(\cdot,0)$ with $\phi(0,0)=1$ is representable as 
$$\phi(n,0)=\sum_{m\in\{0,1,\dots,\infty\}} \mu(\{\q^m\})\q^{mn},$$
where $\mu$ 
is some probability measure on $\{\q^{m}, m\geq 0\}\cup\{0\}$,
if and only if the associated bivariate array satisfying $\phi(h,t)=\q^h\phi(h,t+1)+\phi(h+1,t)$
is nonnegative.
\end{corollary}

A transformation of weights yields $w'_0(h,t)=\q^{h+t}, w_1'(h,t)=1$, which is a generalized Stirling-I triangle,
thus inhomogeneous Bernoulli processes with probability for head $p(n)= \theta/(\theta+\q^n)$ belong to $\calP$.
By Proposition \ref{divser} these measures are not ergodic, see \cite{GO1} for their explicit decomposition 
over the boundary.
A third family of measures are the $\q$-analogues of P{\'o}lya urn processes \cite{GO1}, with transition probabilities
$$p(h,t)= {[\alpha+h]_\q\over [\alpha+\beta+h+t]_\q}, ~~~q(h,t)={[\beta+t]_\q\over[\alpha+\beta+h+t]_\q}\q^{h+\alpha},$$
where $\alpha,\beta>0$ and $[m]_\q:=1+\q+\dots+\q^{m-1}$.

The case $\q>1$ is treated  by a transformation to the weights 
$w_0'(h,t)=1, w_1'(h,t)=\q^{-t}$ and transposition of ${\mathbb Z}_+^2$ about the diagonal, which
establish equivalence of the graphs with parameters $\q$ and $\q^{-1}$.

\section{Generalized Eulerian triangles}

We define a generalized Eulerian triangle to be a weighted Pascal graph with 
$$w_1(h,t)=t+a,~~w_0(h,t)=h+b,$$
where $a,b>0$. The classical Eulerian triangle is the instance $a=b=1$, where
dimensions are Eulerian numbers that count the number of descents in permutations.
Since the balance condition is fulfilled,  there is a natural fully supported measure
$P^*$ with transition probabilities
$$p(h,t)={t+a\over h+t+a+b},~~~q(h,t)= {h+b\over h+t+a+b}.$$

On the other hand, changing weights to 
$w_1(h,t)=(\theta-h+b)(t+a), w_0(h,t)=(\theta+t+a)(h+b)$
for $\theta=m+a$, $m=0,1,\dots$, we obtain $Q_{m,\infty}$ with 
transition probabilities 
$$p(h,t)= {(\theta-h-b)(t+a)\over \theta(t+h+a+b)},~~~
q(h,t)={\theta+t+a)(h+b)\over \theta(t+h+a+b)}.$$
Similarly, the measures $Q_{\infty,m}$ are constructed.

\begin{proposition}
The measures $Q_{m,\infty}, Q_{\infty,m}$ with $m\geq 0$, and $P^*$
comprise the boundary of the generalized Eulerian triangle.
\end{proposition}
\proof 
As $m\to\infty$, the finitely supported measures $Q_{m,\infty}, Q_{\infty,m}$ converge to $P^*$.
By a version of Proposition \ref{discr} and compactness, these measures comprise the sequential boundary
${\rm ext}^\circ\calP$. It remains to be shown that $P^*$ is ergodic, but this 
follows, because otherwise $P^*$  were decompasable in mixture of finitely supported
 elements of ${\rm ext}^\circ\calP$,
which is impossible because  $H_n\to\infty$ and $T_n\to\infty$ $P^*$-a.s.
 \endpf

The process $S$ under $P^*$ is known as the following Friedman's urn model (see \cite{Flaj}, section 2.2).
At time $h+t$ there are $h+t+a+b$ balls in an urn, of which $h+a$ are marked `heads' and $t+a$ `tails'.
A ball is drawn uniformly at random, and returned in the urn together with another ball of the opposite label.
If $a$ or $b$ are not integers, this prescription is to be understood as adding $1$
to the total weight of `heads' with probability
$(t+a)/(h+t+a+b)$. Unlike P{\'o}lya's urn model or the Stirling process in section \ref{alphaL1},
Friedman's urn exhibits concentration of measure in the  form $P^*(H_n/n\to 1/2)=1$.
The latter fact, combined with the criterion of 
ergodicity and the observation that $Q_{h_n,t_n}$ converge weakly to $P^*$ whichever
$h_n\to\infty,t_n\to\infty$,  confirms that $P^*$ is a unique fully supported ergodic measure.

See \cite{Varchenko, adic1, adic2} for variations on the theme 
and other proofs of ergodicity of $P^*$ for the standard Eulerian 
triangle with $a=b=1$.

\def\cprime{$'$} \def\polhk#1{\setbox0=\hbox{#1}{\ooalign{\hidewidth
\lower1.5ex\hbox{`}\hidewidth\crcr\unhbox0}}} \def\cprime{$'$}
\def\cprime{$'$} \def\cprime{$'$}
\def\polhk#1{\setbox0=\hbox{#1}{\ooalign{\hidewidth
\lower1.5ex\hbox{`}\hidewidth\crcr\unhbox0}}} \def\cprime{$'$}
\def\cprime{$'$} \def\polhk#1{\setbox0=\hbox{#1}{\ooalign{\hidewidth
\lower1.5ex\hbox{`}\hidewidth\crcr\unhbox0}}} \def\cprime{$'$}
\def\cprime{$'$} \def\cydot{\leavevmode\raise.4ex\hbox{.}} \def\cprime{$'$}
\def\cprime{$'$} \def\cprime{$'$} \def\cprime{$'$}

\end{document}